\documentclass[11pt]{amsart}
\usepackage[all,knot]{xy}

\usepackage[active]{srcltx}

\usepackage{pb-diagram}

\usepackage[usenames,dvipsnames,svgnames,table]{xcolor}

\usepackage{mathrsfs}
\usepackage{amsmath}
\usepackage{amssymb}
\usepackage{amscd}
\usepackage{amsthm}
\usepackage[latin1]{inputenc}
\usepackage{graphics}
\usepackage{varioref}

\labelformat{enumi}{(#1)}

\newtheorem{thm}{Theorem} [section]
\newtheorem{exe}{Example}[section]
\newtheorem{defin}{Definition} [section]
\newtheorem{remark}{Remark}[section]
\newtheorem{prop}{Proposition}[section]

\newtheorem{lemma}{Lemma}[section]
\def\R{\text{$\mathbb{R}$}}

\def\fis#1{\dot{#1}}
\def\lra{\longrightarrow}
\def\pt#1#2{\tau^{#2}_{#1}}

\def\d1#1#2{\frac{d#1}{d#2}}

\def\tpm{T_p M}
\def\pt#1#2{\tau^{#2}_{#1}}

\def\p1#1#2{\frac{\partial #1}{\partial #2}}

\def\part{a=t_0 < t_1 < \ldots < t_{k-1} < t_{k}=b}

\def\taT{\tilde{T}}

\def\N{\text{$\mathbb{N}$}}
\def\H{\text{$\mathbb{H}$}}
\newcommand{\riem}[2]{({#1},{#2})}

\newcommand{\scalar}{\langle \cdot , \cdot \rangle}
\newcommand{\dd}{\mathrm{d}}
\begin{document}
\author{Leonardo Biliotti and Francesco Mercuri}
\title{Riemannian  Hilbert manifolds}
\address{Universit\`{a} di Parma} \email{leonardo.biliotti@unipr.it}
\address{Universidade Estadual de Campinas (UNICAMP)}
\email{mercuri@ime.unicamp.br}
\thanks{The authors were partially supported by FIRB 2012 ``Geometria differenziale e teoria geometrica delle funzioni'' grant RBFR12W1AQ003.
The first author was partially supported by GNSAGA of INdAM and PRIN ``Variet\`a reali e complesse: geometria, topologia e analisi armonica''.  The second author was partially supported by Fapesp and CNPq (Brazil).} %
 \maketitle
\begin{abstract}
In this article we collect results obtained by the authors jointly with other authors and we discuss old and new ideas. In particular we discuss singularities of the exponential map, completeness and homogeneity for  Riemannian Hilbert quotient manifolds. We also extend a Theorem due to Nomizu and Ozeki to infinite dimensional Riemannian Hilbert manifolds.
\end{abstract}
\section{Introduction}
Let $\H$\ be a Hilbert space. A Riemannian  Hilbert manifold $\riem{M}{\langle \cdot , \cdot \rangle}$,\ RH manifold for short, is a smooth manifold modeled on the Hilbert space $\H$, equipped with an inner product $\langle \cdot , \cdot \rangle_p$ on any tangent space $\tpm$ depending smoothly on $p$ and defining on $\tpm \cong \H$ a norm equivalent to the  one  of $\H$.

The local Riemannian geometry of RH manifolds goes  in the same way as in the finite dimensional case.\ We can prove, just like in the finite dimensional case, the existence and uniqueness of a symmetric connection, compatible with the Riemannian metric, the Levi-Civita connection, characterized by the Koszul formula
\[
\begin{split}
2\langle \nabla _X Y,Z \rangle &= X \langle Y,Z \rangle + Y \langle Z,X  \rangle -
Z \langle X,Y \rangle \\ &+
 \langle [X,Y],Z \rangle + \langle [Z,X],Y \rangle - \langle [Y,Z],X \rangle.
\end{split}
\]
Hence we can define covariant differentiation of a vector field along a smooth curve, parallel translation, geodesics, exponential map, the curvature tensor $R$, its sectional curvature $K$ etc., just like in the finite dimensional case (see \cite{ee,Kl,La}  for details).

The investigation of global properties in infinite dimensional geometry is harder than in the finite dimensional case essentially because of the lack of local compactness. For example, there exist complete RH manifolds with points that cannot be connected by  minimal geodesics, complete connected RH manifolds for which the exponential map is not surjective etc. (see Section \ref{s1}). Moreover, on some RH manifolds one can construct  finite geodesic segments containing  infinitely many conjugate points \cite{Gr}. A complete description of conjugate points along finite geodesic segment is given in \cite{bpet} and similar questions have been studied in \cite{bi2,klmp,klmp2,Mi4,Mi3,Mi,Mi2,mi5,mi6}.

The aim of this survey is to describe  results obtained by the authors jointly with D. Tausk, R. Exel and P. Piccione and others authors \cite{at,bmt,bi,bi2,bi3,bpet,bm,ek,Gr,mc}.  We have tried to avoid technical results in order to make the paper more readable also by non experts in this field. The interested reader will find details and further  results in papers and books quoted in the bibliography.

This paper is organized as follows. In Section \ref{s2}, we investigate complete Riemannian Hilbert manifolds. We extend a Theorem due to Nomizu and Ozeki \cite{no} to Riemannian Hilbert manifolds. We also investigate Hopf-Rinow manifolds, i.e., Riemannian Hilbert manifolds such that there exists minimal geodesic between any two points of $M$, properly discontinuous actions on Riemannian Hilbert manifolds and homogeneity for Riemannian Hilbert quotient manifolds. We also point out that if $M$ has constant sectional curvature then completeness is equivalent to geodesically completeness and there are not non trivial Clifford translations on a Hadamard manifold.
In Section \ref{s1}, following the point of view used by Karcher  \cite{Kar}, we introduce the Jacobi flow, we discuss singularities of the exponential map and the main result proved in \cite{bpet}.
\section{Complete Riemannian Hilbert manifolds}\label{s2}

Let $M$\ be a RH manifold. If $\gamma: [a,b] \subseteq \R \lra M$\ is a piecewise smooth curve, the length of $\gamma$\ is defined, as in the finite dimensional case, $L(\gamma) = \int^b_a \|\dot{\gamma}(t)\| dt$. Then, if $M$\ is connected, we can define a distance function
$$\mathrm{d} (p,q) = inf\{L(\gamma) : \gamma \ \textrm{is a piecewise smooth curve joining}\ \ p\ \textrm{and}\ q\}.
$$
The function $\mathrm{d}$\ is, in fact, a distance and induces the original topology of $M$\ \cite{La,Pa}.
\begin{defin}
We will say that a \ RH manifold $M$\ is {\em complete} if it is complete as a metric space.
\end{defin}
Let $M$ be a Hilbert manifold. A natural question is if there exists a Riemannian metric $\scalar$ such that $(M,\scalar)$ is a complete RH manifold. McAlpin \cite{mc} proved that any separable Hilbert manifold modeled on a separable Hilbert space can be embedded as a closed submanifold of a separable Hilbert space. Hence, if $f:M\lra \H'$\ is such an embedding,  $M$,\ with the induced metric, is a complete RH manifold.  The following result is an extension  to the infinite dimensional case of a  Theorem due to Nomizu and Ozeki \cite{no}.
\begin{thm}
Let $(M,\scalar)$ be a separable RH manifold modeled on a separable Hilbert space. Then there exists a positive smooth function $f:M\lra \R$ such that $(M,f\scalar )$ is a complete RH manifold.
\end{thm}
\begin{proof}
Consider the geodesic ball $ B(p,\epsilon) = \{ q \in M: \mathrm{d} (p,q) < \epsilon\}$. By a result of Ekeland \cite{ek} there exists $\epsilon>0$ such that $\overline{B(p,\epsilon)}$ is a complete metric space. We define
\[
r:M \lra \R, \qquad r(p)=\mathrm{sup}\{r>0:\, \overline{B(p,\epsilon)} \ \mathrm{is\ a\ complete\ metric\ space} \} .
\]
If $r(p)=\infty$ for some $p\in M$ then $(M,\scalar)$ is complete. Hence we may assume $r(p)<+\infty$ for every $p\in M$. We claim $|r(p)-r(q)|\leq \dd(p,q)$ and so it is a continuous function. Indeed, if $\dd (p,q)\geq \mathrm{max}(r(p),r(q))$ then the above inequality holds true. Hence, we may assume without lost of generality that $\dd (p,q)<r(p)$. Pick  $0<\epsilon <\frac{r(p)-d(p,q)}{2}$. The triangle inequality implies $B(q,r(p)- \dd (p,q)-\epsilon)\subset B(p,r(p)-\epsilon)$ and so  $\overline{B(q,\\d (p,q)-r(p)-\epsilon)} \subset \overline{B(p,r(p)-\epsilon)}$. Hence $r(q)\geq r(p)- \dd (p,q)$ and so $|r(p)-r(q)|\leq \dd (p,q)$. Applying a result of \cite{lj}, see also \cite{wul}, there exists a smooth function $f:M \lra \R$ such that $f(x)> \frac{1}{r(x)}$ for any $x\in M$. Pick $\scalar' (x)=f^2 (x) \scalar (x)$. Then $(M,\scalar')$ is a RH manifold. We denote by $\dd '$ the distance defined by $\scalar'$.

Let $x,y\in M$ and let $\gamma:[0,1] \lra M$ be a piecewise smooth curve joining $x$ and $y$. We denote by $L$, respectively $L'$, be the length of $\gamma$ with respect to $\langle \cdot,\cdot \rangle$, respectively the length of $\gamma$ with respect to $\langle \cdot,\cdot \rangle'$. Then
\[
L'=\int_0^1 f(\gamma(t)) \langle \fis \gamma (t), \fis \gamma (t) \rangle^{1/2} \mathrm d t \geq f(\gamma(c)) L > \frac{1}{r(\gamma(c))} L
\]
where $0\leq c \leq 1$. Since $r(\gamma(c))\leq r(x)+d(x,\gamma(c))\leq r(x)+L$, it follows $L'>\frac{L}{r(x)+L}$. Therefore, as in \cite{no}, for any $0<\epsilon< 1$ and for any $x\in M$, we get
$B^{\scalar'}(x,\frac{1}{3-\epsilon})$ is contained in $B(x,\frac{r(x)}{2-\epsilon})$.  Hence $\overline{B^{\scalar'} (x,\frac{1}{3-\epsilon})}$ is a complete metric space, with respect to $\dd$. We claim that $\overline{B^{\scalar'} (x,\frac{1}{3})}$ is a complete metric space with respect to $\dd'$ as well.

Let $\{x_n \}_{n\in \N}$ be a Cauchy sequence of $\overline{B^{\scalar'} (x,\frac{1}{3})}$ with respect to $\dd'$. Let $0<\epsilon<\frac{2}{21}$. Then there exists $n_o$ such that for any $n,m\geq n_o$ we get $\dd'(x_n,x_m)\leq \frac{\epsilon}{4}$. We claim that if $\gamma:[0,1] \lra  M$ is a curve joining $x_n$ and $x_m$, for any $n,m\geq n_o$,  such that $L(\gamma)<\frac{\epsilon}{2}$, then $\gamma([0,1])\subset B(x,\frac{3r(x)}{4})$. Indeed, let $t\in[0,1]$. Then
\[
\dd'(\gamma(t),x)\leq \dd' (\gamma(t),x_n)+\dd' (x_n,x_m)+ \dd' (x_m,x)< \frac{\epsilon}{2}+\frac{\epsilon}{4} +\frac{1}{3}<\frac{1}{3}+\epsilon=\frac{1}{3-\epsilon'},
\]
where $\epsilon'=\frac{9\epsilon}{1+3\epsilon}$. Hence
$
\dd' (\gamma(t),x) < \frac{1}{3-\epsilon'}$ and so $\dd(\gamma(t),x)< \frac{r(x)}{2-\epsilon'}<\frac{3r(x)}{4}.$

Now,
$
L'\geq \frac{1}{r(\gamma(c))} L,
$
for some $0\leq c \leq 1$.
Since $d(\gamma(c),x) < \frac{3r(x)}{4}$, it follows $r(\gamma(c))\leq r(x)+d(x,\gamma(c))\leq r(x)+\frac{3r(x)}{4}=K_o$ and so
$
L'\geq \frac{1}{K_o} L\geq \frac{1}{K_o} \dd(x_n,x_m).
$
Hence
$
\dd'(x_n,x_m)\geq \frac{1}{K_o} \dd (x_n,x_m)
$
and so $\{x_n\}_{n\geq n_o}$ is a Cauchy sequence of $\overline{B(x,\frac{3 r(x)}{4})}$ with respect to $\dd$. Therefore  it converges proving $\overline{B^{\scalar'}(x,\frac{1}{3})}$ is complete with respect to $\dd'$, for any $x\in M$.

Let $\{x_n \}_{n\in \N}$ be a Cauchy sequence with respect to the distance $\dd '$.
Then there exists $n_o$ such that $x_n \in \overline{B^{\scalar'}(x_{n_o},\frac{1}{3})}$ for $n\geq n_o$. Hence $\{x_n \}_{n\geq n_o}$  is a Cauchy sequence of $\overline{B^{\scalar'}(x_{n_o},\frac{1}{3})}$ and so it converges.
\end{proof}
\begin{remark}
In \cite{no} the authors consider the function $r(x)$ to be the supremum of positive numbers $r$ such that the neighborhood $B(x,r)$ is relative compact. This function does not work if $M$ has infinite dimension due of the lack of the local compactness. Moreover, in the finite dimensional case $\overline{B^{\scalar'} (x,\frac{1}{3})}$ is a complete metric space with respect to $d'$ since it is contained in $B(x,\frac{r(x)}{2})$ and so it is compact. In our case, we have to check directly that $\overline{B^{\scalar'} (x,\frac{1}{3}})$ is complete .
\end{remark}
If $M$ is a connected finite dimensional RH manifold, then $M$\  is  complete if and only if it is \emph{geodesically complete} at some point $p\in M$, i.e., there exists $p\in M$ such that the maximal interval of definition of any geodesics starting at $p$ is all of $\R$  and so the exponential map $\exp_p$ is defined on all of $\tpm$ . This also implies that the exponential map $\exp_q$ is defined in all of $T_q M$ for any $q\in M$ and any two points can be joined by a minimal geodesic.
These facts are not true, in general, for  infinite dimensional RH manifolds. The following example is due to Grossman \cite{Gr}.
\begin{exe} Let $\H$\ be a separable Hilbert space with an orthonormal basis $\{e_i, i\in \N$\}. consider
$$
M= \{\sum_{i=1}^{\infty} x_i e_i \in \H  :\ x_1^2 +
\sum_{i=2}^{\infty}(1- \frac{1}{i})^2  x_i^2 =1\}.
$$
Then $M$\
is a complete RH manifold such that $e_1$ and $ -e_1 \in M$ can be connected by infinitely many geodesics but there are not a minimal geodesics between the two points.
\end{exe}
\begin{remark} Atkin \cite{at} modified the above example to construct a complete RH manifold such that the exponential map at some point fails to be surjective.
\end{remark}
 On the other hand the following result holds.
\begin{thm}\label{complete}
Let $M$ be a complete RH manifold and $p\in M$. Then the exponential map $\exp_p$ is defined on all of $\tpm$. Moreover, the set $\mathcal M_p =\{q\in M:$ there exists a unique minimal geodesic joining $p$ and $q$ $\}$ is dense in $M$
\end{thm}
The first part of the Theorem can be proved as in the finite dimensional case. The second part is a result due to Ekeland \cite{ek}. He proved
$\mathcal M_p $ contained a countable intersection of open and dense subsets of $M$. By the Baire's Theorem  it follows $\mathcal M_p$ is dense.

The next result proves that a RH manifold of constant sectional curvature which is geodesically complete it
is also complete.
\begin{prop}
Let $M$ be a RH manifold of constant sectional curvature $K_o$.
Then $M$ is a complete RH manifold if and only there exists $p\in M$ such that $\exp_p$ is defined in all of $\tpm$.
\end{prop}
\begin{proof}
By Theorem \ref{complete}, completeness implies geodesically completeness. Vice-versa,
if the sectional curvature is non positive then geodesic completeness is equivalent to completeness.
This is a consequence of a version of the Cartan-Hadamard Theorem due to McAlpin \cite{mc} and Grossman \cite{Gr,La}.
Hence we may assume  $K_o>0$.
Let $p\in M$ and let $S_{\sqrt{K_o}} (\tpm \times \R)$ the sphere of $\tpm\times \R$ of radius $\frac{1}{\sqrt{K_o}}$. Let $N=(0,\frac{1}{\sqrt{K_o}})\in S_{\sqrt{K_o}} (\H\times \R) $ and let $T:T_{N} S_{\sqrt{K_o}} (\H\times \R) \lra \tpm$ be an isometry. By Proposition \ref{p1} and a Theorem of Cartan \cite[Theorem 1.12.8]{Kl}, the map
\[
F=\exp_p \circ T \circ \exp_N^{-1} :   S_{\sqrt{K_o}} (\H\times \R) \setminus\{-N \} \lra M
\]
is a local isometry. Let $v\in T_{N } S_{\sqrt{K_o}} (\H\times \R)$ be a unit vector. Then $\gamma^v (t)=F(\exp_{N}(tv))$ is a geodesic in $M$. Let $q(v)=\gamma^v (\pi)$. It is easy to see that $q(v)=q(w)$ for any unit vector $w\in T S_{\sqrt{K_o}} (\H\times \R)$. Hence we may extend $F:S_{\sqrt{K_o}} (\H\times \R) \lra M$ and it is easy to check that it is still an isometry. Since $S_{\sqrt{K_o}} (\H\times \R)$
is complete, by \cite[Theorem 6.9 p. 228]{La} we get $F$ is a Riemannian covering map, and so $F$ is surjective, and $M$ is complete.
\end{proof}
\begin{defin}
A Hopf-Rinow  manifold is a complete RH manifold such that any two points $x,y\in M$ can be joined by a minimal geodesic.
\end{defin}
The unit sphere $S(\H)$ is Hopf-Rinow. The Stiefel manifolds of orthonormal $p$ frames in a Hilbert space $\H$  and the Grassmann manifolds  of $p$ subspaces of $\H$ are Hopf-Rinow manifolds \cite{bm,hm}. These manifolds are homogeneous, i.e., the isometry group acts transitively on $M$. It is easy to see that homogeneity implies completeness \cite{bm} but it does not imply the existence of path of minimal length between two points. We also point out that the isometry group of a complete RH manifold can be turned in a Banach Lie group and its Lie algebra is given by the Killing vector fields, i.e.,  vector fields $X$ such that $L_X \langle \cdot , \cdot \rangle=0$. Moreover the natural action of the isometry group on $M$ is smooth (see \cite{klotz}).

In \cite{bi,bm} properly isometric discontinuous actions on the unit sphere of a Hilbert space $\H$ and on the Stiefel and Grassmannian manifolds are studied. We recall that a group $\Gamma$ of isometries acts properly discontinuously on $M$ if for any $f\in \Gamma$, the condition $f (x)=x$ for some $x\in M$ implies $f=\mathrm{e}$ and the orbit throughout any element $x\in M$ is closed and discrete \cite{Kn}.
We completely classify properly discontinuous actions of a finitely generated abelian group on the unit sphere of a separable Hilbert space and we give new examples of complete RH manifolds, respectively K\"ahler RH manifolds, with non negative and non positive sectional curvature with infinite fundamental group, respectively with non negative holomorphic sectional curvature with infinite fundamental group (\cite{bi,bm}).
These new examples of RH  manifolds are Hopf-Rinow manifolds due the following simple fact.
\begin{prop}
Let $M$ be a Hopf-Rinow manifold. Let $\Gamma$ be a group acting isometrically and properly discontinuously on $M$. Then $M/\Gamma$ is also Hopf-Rinow.
\end{prop}
\begin{proof}
Since $\Gamma$ acts isometrically and properly discontinuously on $M$, it follows that $M/\Gamma$ admits a Riemannian metric such that $M/\Gamma$ is complete and $\pi:M \lra M/\Gamma$ is a Riemannian covering map \cite{bi,La}. Let $p,q\in M/\Gamma$. Since $\Gamma$ acts properly discontinuously on $M$, then both $\pi^{-1}(p)$ and $\pi^{-1}(q)$ are $\Gamma$ orbits, and also closed and discrete subsets of $M$ \cite{Kn}. Hence given $z\in \pi^{-1}(p)$, there exists a unique $w\in \pi^{-1}(q)$ such that $\dd (z,w)\leq \dd (r,s)$ for every $r\in \pi^{-1}(p)$ and $s\in \pi^{-1}(q)$, i.e.,
$\dd (z,w)= \dd (\pi^{-1}(p),\pi^{-1}(q))$.
Let $\gamma$ be a minimal geodesic joining $z$ and $w$. We claim that $\pi \circ \gamma$ is a minimal geodesic. Since $\pi$ is a Riemannian covering map, then $\dd (p,q)\leq L(\pi \circ \gamma)=L(\gamma)=\dd (z,w)$. On the other hand pick a sequence $\gamma_n:[0,1] \lra M/\Gamma$ joining $p$ and $q$ such that $\lim_{n\mapsto +\infty} L(\gamma_n)=\dd (p,q)$. Since $\pi$ is a Riemannian covering map there exists a lift $\tilde \gamma_n$ starting at $z$ satisfying $L(\gamma_n)=L(\tilde \gamma_n)$. Therefore
\[
L(\gamma)=\dd (z,w)\leq L(\gamma_n) \mapsto \dd (p,q),
\]
and so $L(\pi \circ \gamma)=\dd (p,q)$.
\end{proof}
In \cite{bm} we prove a homogeneity result for  Riemannian Hilbert manifolds of constant sectional curvature. In finite dimension this result was proved by Wolf \cite{wolf1,wolf}.

An isometry $f:M \lra M$ is called a \emph{Clifford translation} if $\delta_f (x)=\dd (x,f(x))$ is a constant function. As in the finite dimensional case, if $M$ is a homogeneous Riemannian manifold and  $\Gamma$ a group acting on $M$ isometrically and properly discontinuously on $M$, then $M/\Gamma$ is homogeneous if and only if the centralizer of $\Gamma$, that we denote by $\mathrm{Z}(\Gamma)$,  acts transitively on $M$ \cite{bm,wolf}. In particular if $M/\Gamma$ is homogeneous then any element $g\in \Gamma$ is a Clifford translation. Indeed,
\[
\dd (x,g(x))=\dd (h(x),hg(x))=\dd (h(x),g(h(x))),
\]
for any $h\in \mathrm{Z}(\Gamma))$. Hence if $\mathrm{Z}(\Gamma)$ acts transitively on $M$ we get $f$ is a Clifford translation.

In the finite dimensional case, the homogeneity conjecture says that if $M$ is a homogeneous simply connected Riemannian manifold then $M/\Gamma$   is homogeneous if and only if all the elements of $\Gamma$ are Clifford translations. We point out that the conjecture is true for locally homogeneous symmetric spaces \cite{wolf2} and also for locally homogeneous Finsler symmetric spaces \cite{dw}. In \cite{bm} we proved the homogeneity conjecture for complete RH manifolds of constant sectional curvature. We leave the investigation of locally  homogeneous symmetric space of infinite dimension for future investigation (see  \cite{duc,klo,La} for basic references of symmetric space in infinite dimension.) The following result proves there are not non trivial Clifford translations on a Hadamard manifold, i.e., a simply connected Riemannian Hilbert manifold with negative sectional curvature.
\begin{prop}
Let $M$ be a simply connected RH   manifold of negative sectional curvature. If $f:M \lra M$ is a Clifford translation
then $f=\mathrm{Id}$.
\end{prop}
\begin{proof}
Assume $f(p)\neq p$ for some $p\in M$,\ hence for every $p \in M$. By  Cartan-Hadamard Theorem $M$ is a Hopf-Rinow manifold and so by Lemma 5.2 p. 448 in \cite{bm}, see also \cite{ozul}, $f$ preserves the minimal geodesic, that we denote by $\gamma_p$, joining  $p$ and $f(p)$.  Let $p\in M$ and let $\theta$ be a geodesic different from $\gamma_p$. As in the Proof of Theorem $1$ p. 16 in \cite{wolf3}, one can prove that the union $\gamma_{\theta(t)}$ is a flat totally geodesic surface which is a contradiction.
\end{proof}

\section{Jacobi fields and conjugate points}\label{s1}
Let $M$\ be a RH manifold and let $\gamma:[0,b)\lra M$ be a geodesic with $\gamma(0) = p$. Without lost of generality we assume that $\gamma(t)=\exp_p (tv)$, with $\| v \| = 1$.  A Jacobi field along $\gamma$ is a smooth vector field $J$ along $\gamma$ satisfying
$$
\nabla_{\fis \gamma (t)} \nabla_{\fis \gamma (t)} J (t)+R(\fis \gamma (t),J(t))J(t)=0.
$$
In the sequel we will denote by $J'(t)$\ the covariant derivative $\nabla_{\fis \gamma (t)} J (t)$. If $J_1$ and $J_2$ are Jacobi fields along $\gamma$, then
\begin{gather}\label{formula}
\langle J_1'(t),J_2 (t) \rangle - \langle J_1 (t), J_2 ' (t) \rangle =\mathrm{Constant}.
\end{gather}
This formula is due to Ambrose (see \cite{La}). The Jacobi field along $\gamma$ satisfying $J(0)=0$ and $J'(0)=\nabla_{\fis \gamma (t)} J(0)=w$ is given by $J(t)=(\mathrm d \exp_p )_{tv} (tw)$. Hence
$(\mathrm d \exp_p )_v (w)=0$ if and only if there exists a Jacobi field  $J$ along $\gamma(t)$ such that $J(0)=0$ and $J(1)=0$.

In infinite dimension there exist two types of singularities of the exponential map.
\begin{defin}
We will say that $q=\gamma(t_o)$, $t_o \in (0,b)$, is
 \begin{itemize}
 \item \emph{monoconjugate} to $p$\ along $\gamma$ if $(\mathrm d \exp_p )_{t_o v}$ is not injective,
 \item \emph{epiconjugate}, to $p$ along $\gamma$ if $(\mathrm d \exp_p )_{t_o v}$ is not surjective.
\end{itemize}
  We also say $q=\gamma(t_o)$
is conjugate of $p$ along $\gamma$ if $(\mathrm d \exp_p )_{t_o v}$ is not an isomorphism and $t_o \in (0,b)$ is a conjugate, monoconjugate, respectively epiconjugate instant if $\gamma(t_o)$ is conjugate, monoconjugate, respectively epiconjugate of $p$ along $\gamma$.
\end{defin}

Let $ \pt{t}{s} : T_{\gamma(t)} M
\lra T_{\gamma (s)} M$ be the isometry between the tangent
spaces given by the parallel transport along the geodesic $\gamma$. The following result is easy to check.
\begin{lemma}\label{l1}
If $V: [0,b) \lra \tpm$,\ then
$
\nabla_{\fis \gamma (t)} \pt{0}{t} (V(t))=\pt{0}{t}(\fis V (t)).
$
\end{lemma}
By the above Lemma, a Jacobi field along $\gamma$ such that $J(0)=0$ is given by $J(t)=\pt{0}{t} ( T(t)(V))$, where $V\in \tpm$,  and $T(t)$ is a family of endomorphism of $\tpm$ satisfying
\[
\left \{ \begin{array}{l}
T''(t)\ +\ R_t(T(t))=0;\\
T(0)=0,\ T'(0)=\mathrm{Id},
\end{array}
\right.
\]
where
$
R_t:\tpm    \lra \tpm$ is a one parameter family od endomorphism of $\tpm$ defined by  $R_t(X)=\pt{t}{0}(R(\pt{0}{t}(X),\fis{\gamma}(t))\fis{\gamma}(t)).
$
We call the above differential equation the {\em Jacobi flow } of $\gamma$.
\begin{exe}\label{Jacobi-constat}
Assume that $M$ is a RH manifold with constant sectional curvature $K_o$. Then
\[
T(t)(w)=\left\{\begin{array}{ll}
\frac{\sinh (t\sqrt{-K_o})}{\sqrt{-K_o}} w & K_o<0 \\
t w & K_o=0 \\
\frac{\sin (t\sqrt{K_o})}{\sqrt{K_o}}  & K_o >0
\end{array}\right.
\]
\end{exe}
Karcher used the Jacobi flow to get Jacobi fields estimates \cite{Kar}. By standard properties of the curvature, it follows $R_t$ is a symmetric endomorphism of $\tpm$. Since $\pt{0}{t} \circ T(t)=t (\mathrm d \exp_p)_{tv}$ we may thus equivalently state the definitions of monoconjugate, epiconjugate in terms of injectivity, respectively surjectivity of $T(t)$. Moreover, conjugate instants are also discussed in terms of  Lagrangian curves \cite{bpet}. Indeed,
the Hilbert space $\tpm \times \tpm$ has a natural symplectic structure given by $\omega ((X,Y),(Z,W))=\langle X,W \rangle - \langle Y,Z \rangle$. It is easy to check that $\Psi(t):\tpm \times \tpm \lra \tpm \times \tpm$ defined by $\Psi(t)(X,Y)=(\pt{t}{0} (J(t)),\pt{t}{0}(J'(t)))$, where $J(t)$ is the Jacobi field along $\gamma$ such that $J(0)=X$ and $J'(0)=Y$, is a symplectomporhism of $(\tpm \times \tpm, \omega)$. Then $E_t=\Phi_t (\{0\} \times \tpm)$ is a curve of Lagrangian subspaces of $\tpm \times \tpm$. Moreover $t_o \in (0,b)$ is a monoconjugate instant, respectively a epiconjugate instant, if and only if $E_t \cap (\{0\}\times \tpm) \neq \{0\}$, respectively if and only if $E_t  + (\{0\} \times \tpm) \neq \tpm \times \tpm$.

Let $t_o\in (0,b)$. We compute the transpose of $T(t_o)$.
Let $J_1 (t)=\pt{0}{t}(T(t)(v))$ and let $u \in \tpm$. Let $J_2$ be the Jacobi field along the geodesic $\gamma$ such that
$
J_2(t_o)=0, \ \nabla_{\fis \gamma(t_o)} J_2 (t_o)=\pt{0}{t_o}(u).
$
By (\ref{formula}), we have
$
\langle J_1(t_o),J_2'(t_o)\rangle=\langle J_1'(0),J_2 (0)\rangle
$
and so
$
\langle T(t_o)(v),u \rangle=\langle v, \pt{t_o}{0} (J_2 (t_0)) \rangle.
$
Let $\overline\gamma(t)=\gamma(t_o-t)$  and let
\[
\left \{ \begin{array}{l}
\taT''(t)\ +\ R_t (\taT(t))=0;\\
\taT(0)=0,\ \taT'(0)=id,
\end{array}
\right.
\]
be the Jacobi flow along $\overline{\gamma}$. Summing up we have proved that $T^* (t_o)=\pt{t_o}{0} \circ \taT (t_o) \circ  \circ \pt{0}{t_o}$. As a corollary, keeping in mind  Example \ref{Jacobi-constat}, we get the following result.
\begin{prop} \label{p1}
The kernel of $ T(t_o)$ and the kernel of $T^* (t_o)$ are isomorphic. Hence a monoconjugate point is also epiconjugate. Moreover, if $M$ has constant sectional curvature $K_o$, then $T(t)$ is an isomorphism for any $t>0$ whether $K_o \leq 0$, and $T(t)$ is an isomorphism for $0<t<\frac{\pi}{\sqrt{K_o}}$ whether $K_o>0$.
\end{prop}
The above result was proven by McAlpin \cite{mc} and Grossmann in \cite{Gr}. Since both Rauch and Berger Comparison Theorems work for RH   manifolds \cite{bi3,La}, they also work for a weak Riemannian Hilbert manifold \cite{bi2}, the second part of Proposition \ref{p1} holds for any RH manifold with negative sectional curvature and for any RH manifold with sectional curvature bounded above for a constant $K_o>0$.

Proposition \ref{p1} implies that  if $\mathrm{Im}\, T(t_o)$ is closed then
monoconjugate implies epiconjugate and vice-versa. This holds, for example, if $\exp_p$ is Fredholm.
We recall that a smooth map between Hilbert manifolds $f:M \lra N$ is called
Fredholm if for each $p \in M$ the derivative $(\mathrm d f)_p: T_p M \lra
T_{f(p)} N$ is a Fredholm operator. If $M$ is connected then the
${\rm ind}\ (df)_p$ is independent of $p$, and  one defines
the index of $f$ by setting ${\rm ind} (f)={\rm ind} (df)_p$ (see
\cite{et,Ru}).
Misiolek proved that the exponential map of a free loop space with its natural Riemannian metric is Fredholm \cite{Mi}. Misiolek also pointed out that if the curvature is a compact operator, i.e., for any $X\in \tpm$, the map $Z \mapsto R(Z,X)X$
is a compact operator, then $T(t)$ is Fredholm of index zero and so the exponential map is Fredholm as well \cite{Mi2}. Indeed,
\[
T(t)=tId -\int_{0}^t (\int_0^h R_s (T(s)) ds)dh
\]
and so $T(t)=tId +K(t)$ where $K(t)$ is a compact operator. Hence $T(t)$  is Fredholm \cite{Ru} and so $\exp_p$ is Fredholm.

It
is convenient to introduce the notion of {\em strictly epiconjugate\/} instant, to denote an instant $t\in\left]0,b\right[$
for which the range of $T(t)$ fails to be closed. Unlike finite-dimensional Riemannian geometry, conjugate instants
can accumulate. The classical example of this phenomenon is given by an infinite dimensional ellipsoid
in $\ell^2$ whose axes form a non discrete subset of the real line given by Grossman (\cite{Gr}).

Let
$M= \{ x \in \ell^2 :\ x^2_1\ +\ x^2_2 \ + \ \sum_{i=3}^{\infty} (1-\frac{1}{i})^4 x^2_i=1 \}$. $M$ is a closed submanifold of $\ell^2$
and the curve $
\gamma(t)= \cos t  e_1 \ + \ \sin t  e_2
$
is a geodesic of $M$ since it is the set of fixed points of the isometry
\[
F(\sum_{i=1}^\infty x_i e_i)= x_1 e_1 + x_2 e_2 +\sum_{i=3}^\infty (-x_i)e_i.
\]
For any $k\geq 3$,
$
E_k:= \{ x^2_1\ +\ x^2_2 \ +\ (1-\frac{1}{k})^4 x^2_k =1\ \} \hookrightarrow M
$
is a totally geodesic submanifold of $M$ since it is the fixed points set of the isometry
$
F(\sum_{i=1}^\infty x_i e_i)= x_1 e_1 - x_2 e_2 + x_k e_k +\sum_{i=3, i\neq k}^\infty (-x_i)e_i.
$
Hence $K(\fis{\gamma}(s),e_k)=(1-\frac{1}{k})^2$, $J_k(t)=\sin (t(1-\frac{1}{k})) e_k$  is the Jacobi field along $\gamma$ satisfying $J(0)=0$ and $J'(0)=e_k$. Consider $q_k=\frac{k\pi}{k-1}$.\ Then $q_k$ is a sequence of monoconjugate instant such that $\lim_{k\mapsto \infty} q_k =\pi$. We claim that $-e_1=\gamma(\pi)$ is a strictly epiconjugate point. Indeed,
\[
T(\pi) (e_2+ \sum_{k=3}^{\infty} b_k e_k )=
e_2+ \sum_{k=3}^{\infty} b_k \sin ((\frac{k-1}{k}) \pi) e_k
\]
 which implies $T(\pi)$ is injective. On the other hand $\sum_{i=3}^\infty \frac{1}{k} e_k$ does not lie in $\mathrm{Im}\, T(\pi)$ and so $\gamma(\pi)$ is strictly epiconjugate. Indeed if $\sum_{i=3}^\infty \frac{1}{k} e_k \in \mathrm{Im}\, T(\pi)$ then $\sum_{k=3}^\infty \frac{1}{k} e_k=\sum_{k=3}^{\infty} b_k \sin ((1-\frac{1}{k}) \pi) e_k$ and so
$
-\sin (\pi\frac{1}{k}) b_k= \frac{1}{k}.
$
Hence
\[
\lim_{k\mapsto +\infty} b_k=-\lim_{k\mapsto +\infty} k\sin(\pi \frac{1}{k})=-\pi
\]
which is a contradiction. Hence $\gamma(\pi)$ is a strictly epiconjugate point along $\gamma$ and it is an accumulation point of  sequence of monoconjugate points.

In \cite{bpet} the authors give a complete characterization of the
conjugate instants along a geodesic. In particular the set of conjugate instants is closed and the set of strictly epiconjugate points are limit of conjugate points as before. Hence if there is no strictly epiconjugate
instant along $\gamma$ then the set of conjugate instants along any compact
segment of $\gamma$ is finite. Under these circumstances a Morse Index Theorem for geodesics in RH manifolds holds true.


\begin{thebibliography}{10}
\bibitem{at}
{\sc C.~J. Atkin:}
\newblock The Hopf-Rinow theorem is false in infinite dimensions.
\newblock  Bull. London Math. Soc. \textbf{7}, 261-266 (1975).
%
\bibitem{bmt}
{\sc L. Biliotti, F. Mercuri,D. Tausk:}
\newblock A note on tensor fields in Hilbert spaces.
\newblock An. Acad. Brasil. Ci\^enc. \textbf{74}, 207-210 (2002).
%
\bibitem{bi}
{\sc L. Biliotti:}
\newblock Properly Discontinuous isometric actions on the unith sphere of infinite dimensional Hilbert spaces.
\newblock Ann. Global Anal. Geom. \textbf{26}, 385-395 (2004).
%
\bibitem{bi2}
{\sc L. Biliotti:}
\newblock Exponential map of a weak Riemannian Hilbert manifold.
\newblock Illinois J. Math. \textbf{48}, 1191-1206 (2004) .
%
\bibitem{bi3}
{\sc L. Biliotti:}
\newblock Some results on infinite dimensional Riemannian geometry.
\newblock Acta Sci. Math. (Szeged) \textbf{72}, 387-405 (2006).
%
\bibitem{bpet}
{\sc L. Biliotti, R. Exel, P. Piccione, D. Tausk:}
\newblock On the singularities of the exponential map in infinite dimensional Riemannian manifolds.
\newblock Math. Ann. \textbf{336}, 247-267 (2006).
%
\bibitem{bm}
{\sc L. Biliotti, F. Mercuri:}
\newblock Properly discontinuous actions on Hilbert manifolds
\newblock Bull. Braz. Math. Soc. \textbf{45}, 433-452 (2014).
%
%
\bibitem{dw}
{\sc S. Deng, J. A. Wolf:}
\newblock Locally symmetric homogeneous Finsler spaces.
\newblock Int. Math. Res. Not. \textbf{18},  4223-4242 (2013).
%
\bibitem{duc}
{\sc B. Duchesne:}
\newblock Infinite dimensional Riemannian symmetric spaces with fixed-sign curvature operator.
\newblock Ann. Inst. Fourier (Grenoble) \textbf{65}, 211-244 (2015).
%
\bibitem{ee}
{\sc J. Eells:}
\newblock A setting for global analysis.
\newblock Bull. Amer. Math. Soc. \textbf{72}, 751-807 (1966).
%
%
\bibitem{ek}
{\sc I. Ekeland:}
\newblock{The Hopf-Rinow Theorem in infinite dimension.}
\newblock J. Diff. Geometry \textbf{13}, 287-301 (1978).
%
\bibitem{et}
{\sc D. Elworthy, A. Tromba:}
\newblock Differential structures and Fredholm maps on Banach manifolds.
\newblock  Proc. Sympos. Pure Math. \textbf{15}, 45-94 (1970).
%
\bibitem{Gr}
{\sc N. Grossman:}
\newblock Hilbert manifolds without epiconjugate points.
\newblock Proc. Amer. Math. Soc \textbf{16}, 1365-1371 (1965).
%
\bibitem{hm}
{\sc P. Harms, A. Mennucci:}
\newblock Geodesics in infinite dimensional Stiefel and Grassmannian ma\-ni\-folds.
\newblock  C.R. Math. Acad. Sci. Paris \textbf{350}, 773-776 (2013).
%
\bibitem{Kar}
{\sc H. Karcher:}
\newblock Riemannian center of mass and mollifer smoothing.
\newblock  Comm. Pure Appl. Math. \textbf{30}, 509-541 (1977).
%
\bibitem{klmp}
{\sc B. Khesin, J. Lenells, G. Misiolek, S.C. Preston:}
\newblock Curvatures of Sobolev metrics on diffeomorphism groups.
\newblock  Pure Appl. Math. Q. \textbf{9}, 291-332 (2013).
%
\bibitem{klmp2}
{\sc B. Khesin,J. Lenells, G. Misiolek, S. C.Preston:}
\newblock Geometry of diffeomorphism groups, complete integrability and geometric statistics.
\newblock  Geom. Funct. Anal. \textbf{23}, 334-366 (2013).
%
\bibitem{Kl}
{\sc W. Klingenberg:}
\newblock \textit{Riemannian geometry.}
\newblock De Gruyter studies in Mathemathics, New York, (1982).
%
\bibitem{klo}
{\sc M.  Klotz:}
\newblock Banach symmetric spaces.
arxiv:0911.2089.
%
\bibitem{klotz}
{\sc M.  Klotz:}
\newblock The  automorphism  group  of  a  Banach  principal  bundle  with $\{1\}$-structure.
\newblock Geom. Dedicata \textbf{151}, 161-182 (2011).
%
\bibitem{Kn}
{\sc S. Kobayashi, K. Nomizu:}
\newblock \textit{Foundations of Differential Geometry.}
\newblock Vol. I. Reprint of the 1963 original. Wiley Classics Library. A Wiley-Interscience Publication. John Wiley \& Sons, Inc., New York, (1996).
%
%
\bibitem{La}
{\sc S. Lang:}
\newblock \textit{Fundamentals of Differential Geometry.}
\newblock 3nd edn., Graduate Texts in Mathematics, {\bf 191},
Springer-Verlang, New York, (1999).
%
\bibitem{lj}
{\sc L. Llavona:}
\newblock Approximation methods by regular functions.
\newblock Mediterr. J. Math \textbf{3}, 259-271 (2006).
%
\bibitem{mc}
{\sc J. McAlpin:}
\newblock \textit{Infinite dimensional manifolds and Morse Theory.}
\newblock thesis, Columbia University, (1965).
%
\bibitem{Mi4}
{\sc G. Misiolek:}
\newblock Stability of flows of ideal fluids and the geometry of the group of diffeomorphisms.
\newblock Indiana Univ. Math. J.  {\textbf 42}, 215-235 (1993).
%
\bibitem{Mi3}
{\sc G. Misiolek:}
\newblock Conjugate points in ${\mathcal{D}}_{\mu} (T^2)$.
\newblock Proc. Amer. Math. Soc.  \textbf{124}, 977-982 (1996).
%
\bibitem{Mi}
{\sc G. Misiolek:}
\newblock The exponential map on the free loop space is {F}redholm.
\newblock Geom. Funct. Anal.  \textbf{7}, 1-17 (1997).
%
\bibitem{Mi2}
{\sc G. Misiolek:}
\newblock Exponential maps of {S}obolev metrics on loop groups.
\newblock Proc. Amer. Math. Soc \textbf{127}, 2475-2482 (1999).
%
\bibitem{mi5}
{\sc G. Misiolek:}
\newblock The exponential map near conjugate points in 2D hydrodynamics.
\newblock  Arnold Math. J. \textbf{1}, 243-251 (2015).
%
%
\bibitem{mi6}
{\sc G. Misiolek, S. C. Preston:}
\newblock Fredholm properties of Riemannian exponential maps on diffeomorphism groups.
\newblock Invent. Math. \textbf{179}, 191-227 (2010).
%
\bibitem{no}
{\sc K. Nomizu, H. Ozeki:}
\newblock{The existence of complete Riemannian metrics.}
\newblock Proc. Amer. Math. Soc. \textbf{12}, 889-891 (1961).
%
\bibitem{ozul}
{\sc V. Ozols:}
\newblock Clifford Traslations of symmetric spaces.
\newblock Proc. Amer. Math. Soc. \textbf{44}, 169-175 (1974).
%
\bibitem{Pa}
{\sc R. Palais:}
\newblock Morse Theory on Hilbert manifolds.
\newblock  Topology \textbf{2}, 299-340 (1963).
%
\bibitem{Ru}
{\sc W. Rudin:}
\newblock \emph{Functional Analysis.}
\newblock 2nd edn., International Series in Pure and
 Applied Mathematics. McGraw-Hill, Inc., New York, (1991).
%
%
\bibitem{wolf1}
{\sc A. J. Wolf:}
\newblock Sur la classification des vari\'et\'es riemannienes homog\'enes \'a courbure constante.
\newblock  C. R. Acad. Sci. Paris \textbf{250}, 3443-3445 (1960).
%
\bibitem{wolf2}
{\sc A. J. Wolf:}
\newblock Locally symmetric homogeneous spaces.
\newblock Comment. Math. Helv. \textbf{37}, 65-101 (1962/63)  .
%
\bibitem{wolf3}
{\sc A. J. Wolf:}
\newblock Homogeneity and bounded isometries in manifolds of negative curvature.
\newblock Illinois J. Math. \textbf{8}, 14-18 (1964).
%
\bibitem{wolf}
{\sc Wolf, A. J.:}
\newblock \emph{Spaces of constant curvature.}
\newblock 6end, AMS Chelsea Publishing1 Providence, RI (2011).
%
\bibitem{wul}
{\sc D. Wulbert:}
\newblock Approximation by $C^\infty$ functions.
\newblock Proc. Internat. Sympos. Univ. Texas, Austin, Tex. 217-239 (1973).
%
\end{thebibliography}
\end{document}